\documentclass[12pt, reqno]{amsart}
\usepackage{amssymb, amsmath, amsthm, cancel, hyperref, color, enumerate, verbatim, mathtools}

\usepackage[margin=1in, includeheadfoot]{geometry}

\newtheorem{thm}{Theorem}
\newtheorem{rem}{Remark}
\newtheorem{lem}{Lemma}
\newtheorem{cor}{Corollary}
\newtheorem{prop}{Proposition}

\newcommand{\lcm}{\mathrm{lcm}}
\newcommand{\bul}{\raisebox{0.25ex}{\tiny$\bullet$}}
\newcommand{\leg}[2]{\genfrac{(}{)}{}{}{#1}{#2}} 

\makeatletter 
\let\@@pmod\pmod
\DeclareRobustCommand{\pmod}{\@ifstar\@pmods\@@pmod}
\def\@pmods#1{\mkern4mu({\operator@font mod}\mkern 6mu#1)}
\makeatother
\newcommand{\cphi}[1]{c\phi_{#1}} 
\newcommand{\cphibar}[1]{\overline{c\phi_{#1}}} 

\begin{document}

\title{Congruences for modular forms and generalized Frobenius partitions}
\author{Marie Jameson \and Maggie Wieczorek}

\maketitle
\section*{Abstract}
The partition function is known to exhibit beautiful congruences that are often proved using the theory of modular forms. In this paper, we study the extent to which these congruence results apply to the generalized Frobenius partitions defined by Andrews \cite{Andrews84}.  In particular, we prove that there are infinitely many congruences for $c\phi_k(n)$ modulo $\ell,$ where $\gcd(\ell,6k)=1,$ and we also prove results on the parity of $c\phi_k(n).$  Along the way, we prove results regarding the parity of coefficients of weakly holomorphic modular forms which generalize work of Ono \cite{Ono962}.


\section{Introduction and statement of results}

The partition function $p(n)$ is a beautiful and well-known function, particularly because it enjoys such surprising and delightful congruence properties. In this article, we discuss the extent to which these congruence properties carry over to the generalized Frobenius partitions defined by Andrews \cite{Andrews84}, whose counting function $c\phi_k(n)$ is a generalization of $p(n).$

In order to begin our discussion, we briefly review four of the most well-known properties for the partition function.  First, we have the Ramanujan congruences, which assert that for every nonnegative integer $n$, we have
\begin{align*}
p(5n+4) &\equiv 0\pmod{5}\\
p(7n+5) &\equiv 0\pmod{7}\\
p(11n+6) &\equiv 0\pmod{11}.
\end{align*}
Second, we know (by work of Ahlgren and Boylan \cite{AhlgrenBoylan03}) that the Ramanujan congruences are the only ``simple'' congruences, i.e., congruences of the form
\[p(\ell n+\beta) \equiv 0\pmod{\ell},\]
where $\ell$ is prime and $0\leq \beta< \ell.$ Although this may seem to imply that congruences for $p(n)$ are quite rare, the third well-known property of $p(n)$ is that they occur in great abundance, as long as we relax our restrictions on the shape of the arithmetic progression: for any prime $\ell$ coprime to 6 and positive integer $m$, Ahlgren and Ono \cite{AhlgrenOno01} proved that there exist infinitely many non-nested arithmetic progressions $\{An+B\}$ such that for every positive integer $n$ we have
\[p(An+B)\equiv 0\pmod{\ell^m}.\]
However, the behavior of $p(n)$ modulo 2 and 3 is quite different, essentially because 2 and 3 are the primes which divide the level of the associated modular form.  This leads us to the fourth well-known fact about the partition function, known as Subbarao's Conjecture, which is now a theorem due to Radu \cite{Radu12}: for any arithmetic progression $r\pmod*{t},$ there are infinitely many integers $M\equiv r\pmod*{t}$ for which $p(M)$ is odd, and there are infinitely many integers $N\equiv r\pmod*{t}$ for which $p(N)$ is even.

Now, we turn to the case of generalized Frobenius partitions, which were defined combinatorially by Andrews \cite{Andrews84}. We omit their definition here, and instead are content with the knowledge that their counting function $c\phi_k(n)$ has a generating function that is essentially a modular form (given explicitly in Section \ref{section2}).  When $k=1,$ we obtain the generating function
\[\sum_{n=0}^\infty c\phi_1(n)q^n = \prod_{n=1}^\infty\frac{1}{(1-q^n)} = \sum_{n=0}^\infty p(n)q^n,\]
i.e., $c\phi_1(n)=p(n),$ and thus we may view $c\phi_k(n)$ as a generalization of the partition function. For $k=2$ we have
\[\sum_{n=0}^\infty c\phi_2(n)q^n = \prod_{n=1}^\infty\frac{(1-q^{4n-2})}{(1-q^{2n-1})^4(1-q^{4n})}.\]

In his 1984 work, Andrews found several Ramanujan congruences for $c\phi_k(n)$ where $k>1.$  For example, he proved that when $k$ is prime and $k \nmid n,$ we have that
\[c\phi_k(n)\equiv 0\pmod{k^2},\]
giving at least one Ramanujan congruence for each prime $k$.  When $k=2,$ we also have an additional congruence, giving two Ramanujan congruences \cite{Andrews84}
\begin{align*}
c\phi_2(2n+1) &\equiv 0\pmod{2}\\ 
c\phi_2(5n+3) &\equiv 0 \pmod{5}.
\end{align*}
In this context, we also have that there are only finitely many Ramanujan congruences: Dewar \cite{Dewar11} proved that the congruences above are the only ones for $c\phi_2(n).$ While Dewar's theorem does not apply for all $k$, he notes that his approach of using Tate cycles should apply in more generality.

There are also many other congruences for generalized Frobenius partitions which can be found in the literature.  In fact, there are infinitely many such congruences. This follows by applying a theorem of Treneer \cite{Treneer06}, which gives congruences for all weakly holomorphic modular forms.

\begin{thm}\label{mainthmtreneer}
Let $m$ and $k$ be positive integers and let $\ell$ be prime with $(\ell,6k)=1.$ Then there exist infinitely many non-nested arithmetic progressions $\{An+B\}$ such that for every positive integer $n$ we have
\[c\phi_k(An+B)\equiv 0\pmod{\ell^m}.\]
\end{thm}

\begin{rem}
By combining this theorem with the Chinese Remainder Theorem, we have that for any integer $M$ coprime to $6k,$ there are infinitely many congruences of the form $c\phi_k(An+B)\equiv 0\pmod*{M}.$
\end{rem}

Finally, it remains to consider whether there are congruences for $c\phi_k(n)$ modulo $\ell$, where $\ell \mid 6k.$ A natural first step is to consider $c\phi_2(n)\pmod*{2}.$ In fact, this turns out to be a particularly interesting case, since we already know that Ramanujan congruences exist modulo 2, but Theorem \ref{mainthmtreneer} does not guarantee that there are infinitely many (non-Ramanujan) congruences. However, using the fact that $(1-x^2)^2 \equiv (1-x)^4 \pmod*{4}$, we have that

\begin{align}
\sum_{n=0}^\infty c\phi_2(n)q^n &= \prod_{n=1}^\infty\frac{(1-q^{4n-2})}{(1-q^{2n-1})^4(1-q^{4n})} = \prod_{n=1}^\infty \frac{(1-q^{2n})^5}{(1-q^n)^4(1-q^{4n})^2} \nonumber\\
&\equiv \prod_{n=1}^\infty \frac{1}{(1-q^{2n})} = \sum_{n=0}^\infty p(n/2)q^n\pmod{4}. \label{mod4cong}
\end{align}
In other words, the parity of $c\phi_2(n)$ is completely dictated by the parity of the partition function, and so we define \[\overline{c\phi_2}(n) \coloneqq c\phi_2(n) - p(n/2)\] and set out to now study the parity of $\frac{\overline{c\phi_2}(n)}{4}$. This function was previously studied by Kolitsch and others \cite{Kolitsch89,Sellers94}, and prior results show that the analogue of Subbarao's conjecture is not true. However, we can prove some results in that direction, in the style of Ono \cite{Ono962}.

\begin{thm} \label{mainthm1}
For any arithmetic progression $r \pmod{t},$ there are infinitely many integers $N\equiv r\pmod{t}$ for which $\frac{\overline{c\phi_2}(N)}{4}$ is even.
\end{thm}

\begin{thm} \label{mainthm2}
For any arithmetic progression $r\pmod{t},$ there are infinitely many integers $M\equiv r\pmod{t}$ for which $\frac{\overline{c\phi_2}(M)}{4}$ is odd, provided there is one such $M$. Furthermore, if there does exist an $M\equiv r\pmod{t}$ for which $\frac{\overline{c\phi_2}(M)}{4}$ is odd, then the smallest such $M$ is less than $C_{r,t}$ where
\[C_{r,t}\coloneqq \frac{2^{18+j}\cdot 3^7t^6}{d^2}\prod_{p\mid 6t} \left(1 - \frac{1}{p^2}\right)-2^j,\]
$d\coloneqq \gcd(12r-1,t),$ and $j$ is an integer satisfying $2^j>\frac{t}{12}.$
\end{thm}

In fact, Theorem \ref{mainthm1} and \ref{mainthm2} (as well as their analogues for the partition function, which were proved by Ono \cite{Ono962}) are special cases of a much more general phenomenon, which applies to the coefficients of any weakly holomorphic modular form with algebraic integer coefficients and any arithmetic progression of modulus $t>1.$  We prove these general results along the way (see Theorems \ref{mainthm1gen} and \ref{mainthm2gen}).

This article is organized as follows: in Section \ref{section2}, we outline the connection between $c\phi_k(n)$ and the theory of modular forms and apply results of Treneer \cite{Treneer06} to prove Theorem \ref{mainthmtreneer}, which guarantees congruences modulo $\ell$ for $c\phi_k(n)$ provided $(\ell,6k)=1$. In Section \ref{generalresults}, we step back from partition functions to prove results on the parity of coefficients of weakly holomorphic modular forms in the style of Ono \cite{Ono962}. Lastly, in Section \ref{cors}, we see these results applied to a specific partition function, $\cphibar{2}(n)/4.$


\section{Congruences for $\cphi{k}(n)$ when $(\ell,6k)=1$} \label{section2}

\subsection{Modular forms and the generating function for $\cphi{k}(n)$} In order to prove Theorem \ref{mainthmtreneer}, we first need to study the generating function for $c\phi_k(n)$ and relate it to the theory of modular forms. Andrews found the generating function for $\cphi{k}(n)$ \cite[Theorem 5.2]{Andrews84} to be

\begin{equation} \label{coreq1} 
\sum_{n=0}^{\infty} \cphi{k}(n)q^n = \frac{\displaystyle\sum_{m_1, m_2, \ldots, m_{k-1}\in \mathbb{Z}}q^{Q(m_1, m_2, \ldots, m_{k-1})}}{\displaystyle\prod_{n=1}^{\infty}(1-q^n)^k},
\end{equation}
where $Q(m_1, m_2, \ldots, m_{k-1})$ is given by
\[Q(m_1, m_2, \ldots, m_{k-1}) \coloneqq \sum_{i=1}^{k-1}m_i^2+\sum_{1\leq i<j\leq k-1}m_im_j.\]

Next, we apply the theory of theta functions to show that this generating function is essentially a modular form. Here, we use the following standard notation: for $k, N$ positive integers and $\chi$ a Dirichlet character modulo $N$, we let $M_{\frac{k}{2}}(\Gamma_0(N),\chi)$ \big[\emph{resp. $M^!_{\frac{k}{2}}({\Gamma_0(N)},\chi)$}\big] denote the space of holomorphic [\emph{resp. weakly holomorphic}] modular forms of weight $\frac{k}{2}$ and character $\chi$ for the congruence subgroup $\Gamma_0(N)$ of $\mathrm{SL}_2(\mathbb{Z})$. For more background on modular forms of integer and half integer weight, see \cite{Koblitz93, Miyake, Ono04}. Note also that a similar lemma appears in \cite{ChanWangYang2017}. 

\begin{lem} \label{thelemma}
Let $k$ be an positive integer and let $Q(m_1, m_2, \ldots, m_{k-1})$ be defined as above.
\begin{enumerate}[(a)]
\item If $k$ is odd, then we have that
\begin{equation} \label{coreq2}
\sum_{m_1, m_2, \ldots, m_{k-1}\in \mathbb{Z}}q^{Q(m_1, m_2, \ldots, m_{k-1})} \in M_{\frac{k-1}{2}}(\Gamma_0(k),\chi_k),
\end{equation}
where we set $\chi_k(\bul)\coloneqq\leg{(-1)^{(k-1)/2}k}{\bul}$ for $k$ odd.
\item If $k$ is even, then we have that
\begin{equation} \label{coreq2}
\sum_{m_1, m_2, \ldots, m_{k-1}\in \mathbb{Z}}q^{Q(m_1, m_2, \ldots, m_{k-1})} \in M^!_{\frac{k-1}{2}}({\Gamma_0(2k)},\chi_k),
\end{equation}
where we set $\chi_k(\bul)\coloneqq\leg{2k}{\bul}$ for $k$ even.
\end{enumerate}
\end{lem}

\begin{proof}
Set $r\coloneqq k-1,$ $A \coloneqq \mathbf{I}_{k-1} + \mathbf{1}_{k-1}$ (i.e., the sum of the identity matrix and the all-ones matrix), $h\coloneqq 0, N\coloneqq k,$ $P(x)\coloneqq 1,$ and $v \coloneqq 0.$ Following the notation of Section 4.9 of \cite{Miyake}, we have that \[\theta(z; h, A, N, P) = \sum_{m_1, m_2, \ldots, m_{k-1}\in \mathbb{Z}}q^{Q(m_1, m_2, \ldots, m_{k-1})}.\] Part (a) follows immediately from part (3) of Corollary 4.9.5 of \cite{Miyake}. Part (b) follows from Theorem 4.9.3 (and the remark which follows it) of \cite{Miyake}, since for $\gamma = \begin{pmatrix}a&b\\c&d\end{pmatrix}\in \Gamma_0(2k),$
\[\theta(\gamma z; h, a, N, P) = \leg{2k}{d}\leg{c}{d}^r\varepsilon_d^{-r}(cz+d)^{r/2}\theta(z; h,A,N,P),\]
where as usual $\varepsilon_d \coloneqq \begin{cases}1 & d\equiv 1\pmod*{4}\\i & d\equiv 3\pmod*{4} \end{cases}.$
\end{proof}

\subsection{Treneer's theorems and the proof of Theorem \ref{mainthmtreneer}} Powerful results of Treneer from 2006 \cite{Treneer06} establish congruences for coefficients of certain types of modular forms.

\begin{thm}[Theorem 1.1 of \cite{Treneer06}] \label{treneerthm1.1}
Suppose that $\ell$ is an odd prime, and that $k$ and $m$ are integers with $k$ odd. Let $N$ be a positive integer with $4\mid N$ and $(N,\ell)=1,$ and let $\chi$ be a Dirichlet character modulo $N$. Let $K$ be an algebraic number field with ring of integers $\mathcal{O}_K,$ and suppose $f(z)=\sum a(n)q^n\in M^!_{\frac{k}{2}}({\Gamma_0(N)},\chi)\cap\mathcal{O}_K((q))$. If $m$ is sufficiently large, then for each positive integer $j$, a positive proportion of the primes $P\equiv-1\pmod*{N\ell^j}$ have the property that
\[a(P^3\ell^mn)\equiv0\pmod*{\ell^j}\]
for all $n$ coprime to $P\ell$.
\end{thm}

We can now prove Theorem \ref{mainthmtreneer}, which is essentially a corollary of Theorem \ref{treneerthm1.1}.

\begin{cor} \label{treneercor}
Let $m$ and $k$ be positive integers and let $\ell$ be prime with $(\ell,6k)=1.$ Then if $m$ is sufficiently large, for each positive integer $j$, a positive proportion of the primes $P\equiv-1\pmod*{576k\ell^j}$ have the property that
\[\cphi{k}\left(\frac{P^3\ell^mn+k}{24}\right)\equiv0\pmod*{\ell^j}\]
for all $n$ coprime to $P\ell.$
\end{cor}

\begin{proof}Let $k$ be a positive integer and define
\begin{equation} \label{coreq3}
f(z)\coloneqq \frac{\displaystyle\sum_{m_1, m_2, \ldots, m_{k-1}\in \mathbb{Z}}q^{24Q(m_1, m_2, \ldots, m_{k-1})}}{\eta(24z)^k},
\end{equation}
where $\eta(z) \coloneqq q^{1/24}\prod_{n=1}^\infty(1-q^n)$ is the Dedekind eta-function. It is well-known that $\eta(24z)\in S_{1/2}(\Gamma_0(576),\chi_{12}),$ where $\chi_{12}(\bul)\coloneqq\leg{12}{\bul}.$ Thus by Lemma \ref{thelemma}, we have that $f(z)$ is a weakly holomorphic modular form of level $576k,$ weight $-\frac{1}{2}$ and character $\chi_k\chi_{12}$, i.e.
\[f(z)\in M^!_{-\frac{1}{2}}({\Gamma_0(576k)},\chi_k\chi_{12}).\]
From equations \eqref{coreq1} and \eqref{coreq3} it follows that 
\[f(z)=\sum_{n=-k}^{\infty}\cphi{k}\left(\frac{n+k}{24}\right)q^n.\]

Now, let $\ell\geq5$ be a prime such that $\ell\nmid k$ and $m$ be a positive integer. Theorem \ref{treneerthm1.1} states that if $m$ is sufficiently large, then for each positive integer $j$, a positive proportion of the primes $P\equiv-1\pmod*{576k\ell^j}$ have the property that
\[\cphi{k}\left(\frac{P^3\ell^mn+k}{24}\right)\equiv0\pmod*{\ell^j}\]
for all $n$ coprime to $P\ell$ as desired.
\end{proof}

Note that for each prime $P$ guaranteed by Corollary \ref{treneercor}, we may let $n$ vary in an appropriate arithmetic progression (i.e., one which guarantees that $P^3\ell^m n\equiv -k\pmod*{24}$ and $(n,P\ell)=1$) to obtain Theorem 1.


\section{Parity results for coefficients of modular forms} \label{generalresults}

In this section, we will prove the following general results regarding the parity of coefficients of weakly holomorphic modular forms, which we assume to have algebraic integer coefficients.

\begin{thm} \label{mainthm1gen}
Let $N_0, \alpha, \beta, t$ be integers with $N_0, \alpha, t$ positive, and let
\[\sum_{n=0}^\infty c(n)q^{\alpha n + \beta}\in M_{k}^!(\Gamma_0(N_0),\chi),\]
where $c(n)$ are algebraic integers in some number field. For any arithmetic progression $r \pmod{t},$ there are infinitely many integers $M\equiv r\pmod{t}$ for which $c(M)$ is even.
\end{thm}

\begin{thm} \label{mainthm2gen}
Let $N_0, \alpha, \beta, t$ be integers with $N_0, \alpha$ positive, and $t>1$, and let
\[\sum_{n=0}^\infty c(n)q^{\alpha n + \beta}\in M_{k}^!(\Gamma_0(N_0),\chi),\]
where $c(n)$ are algebraic integers in some number field. For any arithmetic progression $r\pmod{t},$ there are infinitely many integers $M\equiv r\pmod{t}$ for which $c(M)$ is odd, provided there is one such $M$.

Furthermore, if there does exist an $M\equiv r\pmod{t}$ for which $c(M)$ is odd, then the smallest such $M$ is less than $C_{r,t}$ for
\[C_{r,t}\coloneqq \frac{2^j\cdot 12+k}{12\alpha}\left[\frac{N\alpha^2t^2}{d}\right]^2\prod_{p\mid N\alpha t} \left(1 - \frac{1}{p^2}\right)-2^j,\]
where $N\coloneqq\lcm(\alpha t,N_0),$ $d\coloneqq\gcd(\alpha r+\beta,t),$ and $j$ is a sufficiently large integer (as in Proposition \ref{prop1}).
\end{thm}


\subsection{Proof of Theorem \ref{mainthm1gen}} First, we adapt the methods of \cite{Ono962} to prove Theorem \ref{mainthm1gen}.

\begin{prop} \label{prop1}
Let $N_0, \alpha, \beta, t$ be integers with $N_0, \alpha, t$ positive, and let
\[\sum_{n=0}^\infty c(n)q^{\alpha n + \beta}\in M_{k}^!(\Gamma_0(N_0),\chi),\]
where $c(n)$ are algebraic integers in some number field. Then for sufficiently large $j$, we have that 
\[f_t(z)\coloneqq \Delta^{2^j}(\alpha tz) \sum_{n=0}^\infty c(n)q^{\alpha n+\beta}\]
is a cusp form in $S_{12\cdot2^j+k}\left(\Gamma_0(N),\chi\right)$, where $N\coloneqq\lcm(\alpha t,N_0)$. Moreover, the Fourier expansion of $f_t(z)$ modulo 2 can be factored as:
\[f_t(z) \equiv \left(\sum_{n=0}^{\infty}q^{\alpha\cdot2^jt(2n+1)^2}\right) \left(\sum_{n=0}^{\infty}c(n) q^{\alpha n+\beta}\right)\pmod{2}.\] 
\end{prop}

\begin{proof}
Note that $\sum_{n=0}^\infty c(n)q^{\alpha n + \beta}$ is a weakly holomorphic modular form of level $N$ and $\Delta^j(\alpha tz)$ is a cusp form of level $N.$ Thus by choosing $j$ sufficiently large (to ensure vanishing at cusps), it follows that $f_t(z) \in S_{12\cdot 2^j + k}(\Gamma_0(N),\chi).$ This proves the first statement of the theorem.

The second statement of the theorem follows from the well-known fact that (see, for example, \cite{Ono962})
\[\Delta(z)\equiv\sum_{n=0}^{\infty}q^{(2n+1)^2}\pmod{2},\]
together with substitution and the Freshman Binomial Theorem.
\end{proof}

By combining Proposition \ref{prop1} with the following result of Serre \cite{Serre74} concerning the divisibility of coefficients of modular forms, we will be able to prove Theorem \ref{mainthm1gen}.

\begin{cor}[Serre \cite{Serre74}] \label{serrecor}
Let $f(z)$ be a holomorphic modular form of positive integer weight $k$ on some congruence subgroup of $\mathrm{SL}_2(\mathbb{Z})$ with Fourier expansion
\[f(z)=\sum_{n=0}^{\infty}a(n)q^n,\]
where a(n) are algebraic integers in some number field. If $m$ is a positive integer, then
\[a(n)\equiv 0\pmod{m}\]
for almost all $n$ in any given fixed arithmetic progression $r\pmod*{t}$.
\end{cor}

\begin{proof}[Proof of Theorem \ref{mainthm1gen}]
By Proposition \ref{prop1}, we have a cusp form $f_t(z) = \sum_{n=0}^\infty a_t(n)q^{\alpha n+\beta}$ such that 
\[\sum_{n=0}^\infty a_t(n)q^{\alpha n+\beta} \equiv \left(\sum_{n=0}^{\infty}q^{\alpha\cdot2^jt(2n+1)^2}\right) \left(\sum_{n=0}^{\infty}c(n) q^{\alpha n+\beta}\right)\pmod{2}.\]
Thus we have
\begin{equation} \label{mt1eq1}
\sum_{n=0}^{\infty}a_t(n)q^n\equiv \left(\sum_{n=0}^{\infty}q^{2^jt(2n+1)^2}\right) \left(\sum_{n=0}^{\infty} c(n) q^n\right) \pmod{2}
\end{equation}
and, by Corollary \ref{serrecor}, almost all of the coefficients $a_t(n)$ are even.

Now, assume for contradiction that there are finitely many integers $M\equiv r\pmod*{t}$ for which $c(M)$ is even, i.e., that there exists some $n_0$ such that $c(M)$ is odd for all $M\geq n_0$ with $M\equiv r\pmod*{t}$.

For $n\geq n_0$ with $n\equiv r\pmod*{t}$ and $\kappa\equiv1\pmod*{4}$ such that $\kappa>\frac{n}{2^{j+2}t}-1$, we now compare the coefficient of $q^{2^jt\kappa^2+n}$ on each side of equation \eqref{mt1eq1} to obtain 
\begin{equation} \label{mt1eq2}
a_t(2^jt\kappa^2+n)\equiv\sum_{i\geq 1 \text{ odd}} c(2^jt(\kappa^2-i^2)+n)\pmod{2}.
\end{equation}

In order to simplify the right side of equation \eqref{mt1eq2}, note that for odd $i\leq \kappa$, we have that $2^jt(\kappa^2-i^2)+n\geq n\geq n_0$ and $2^jt(\kappa^2-i^2)+n\equiv n\equiv r\pmod*{t}$, so our assumption above guarantees that the summand $c(2^jt(\kappa^2-i^2)+n)$ is odd. On the other hand, for odd $i>\kappa$, we use the fact that $\kappa+1>\frac{n}{2^{j+2}t}$ to obtain
\[2^jt(\kappa^2-i^2)+n\leq 2^jt(\kappa^2-(\kappa+2)^2)+n=-2^{j+2}t(\kappa+1)+n<0,\]
and thus $c(2^jt(\kappa^2-i^2)+n)=0$. Hence, for such $n,\kappa$, equation \eqref{mt1eq2} can be rewritten as
\[a_t(2^jt\kappa^2+n)\equiv\sum_{1\leq i\leq \kappa\text{ odd}}1 = \frac{\kappa+1}{2}\equiv 1\pmod{2}.\]

Thus for $\kappa$ sufficiently large such that $\kappa\equiv 1\pmod*{4},$ the above argument guarantees that for all $M\equiv r\pmod*{t}$ in the interval
\[\big[2^jt\kappa^2+n_0,2^jt(\kappa+2)^2+r-t\big],\] $a_t(M)$ is odd. Note that for distinct values of such $\kappa$, the above intervals are disjoint and the number of such $M$'s in each associated interval is $2^{j+2}(\kappa+1)+\frac{r-n_0}{t}$, where we are assuming without loss of generality that $n_0\equiv r\pmod*{t}$.

Taking into account these intervals for all such values of $\kappa$, we see that if there are only finitely many positive integers $M$ for which $a_t(M)$ is even, then a positive proportion of all $M\equiv r\pmod*{t}$ have $a_t(M)$ odd, contradicting Corollary \ref{serrecor}.
\end{proof}


\subsection{Proof of Theorem \ref{mainthm2gen}}

Now, we prove Theorem \ref{mainthm2gen} using similar methods to those found in \cite{Ono962}. First, we state the following lemmas. 

\begin{lem}[Lemma 2 of \cite{Ono962}] \label{lem2}
Let $f(z)=\sum_{n=0}^\infty a(n)q^n$ be a modular form in $M_k(\Gamma_0(N),\chi)$ and let $d:=\gcd(r,t).$ If $0\leq r<t,$ then \[f_{r,t}(z) = \sum_{n\equiv r\pmod*{t}} a(n)q^n\] is the Fourier expansion of a modular form in $M_k(\Gamma_1(Nt^2/d))$.
\end{lem}

\begin{lem}[Lemma 1 of \cite{Ono962}] \label{lem1}
Let $f(z)=\sum_{n=0}^\infty a(n)q^n$ where the coefficients $a(n)$ are algebraic integers in some number field. Let $s$ and $w$ be positive integers and $b_1, b_2, \ldots, b_s$ distinct non-zero integers. If $m$ is a positive integer and \[f(z) \equiv \sum_{1\leq i\leq s}\sum_{n=0}^\infty a_i(n)q^{w(2n+1)^2+b_i}\pmod{m}\] where $a_i(n) \not\equiv 0\pmod{m}$ for all $n\geq 0,$ then $f(z)$ is not in $M_k(\Gamma_1(N))$ for any pair of positive integers $k$ and $N$.
\end{lem}

Finally, we need the following well-known theorem of Sturm.  Here, for a $q$-series $f(z)=\sum_{n=0}^{\infty}a(n)q^n$ and positive integer $m$, we let $\mathrm{Ord}_m(f)$ denote the smallest integer $n$ such that $a(n)\not\equiv 0\pmod*{m},$ and if no such $n$ exists, we say that $\mathrm{Ord}_m(f)=\infty.$

\begin{lem}[\cite{Sturm87}] \label{sturm}
Let $f(z)=\sum_{n=0}^{\infty}a(n)q^n\in M_k(\Gamma_1(N))$ for some positive integer $N$ with algebraic integer Fourier coefficients from a fixed number field. If $m$ is a positive integer and 
\[\mathrm{Ord}_m(f)>\frac{k}{12}N^2\prod_{p|N}\left(1-\frac{1}{p^2}\right),\]
then $\mathrm{Ord}_m(f)=\infty$, i.e. $a(n)\equiv 0\pmod*{m}$ for all $n$.
\end{lem}

\begin{proof}[Proof of Theorem \ref{mainthm2gen}]
By Proposition \ref{prop1}, we have a cusp form $f_t(z) = \sum_{n=0}^\infty a_t(n)q^{\alpha n+\beta}$ such that 
\[\sum_{n=0}^\infty a_t(n)q^{\alpha n+\beta} \equiv \left(\sum_{n=0}^{\infty}q^{2^j\alpha t(2n+1)^2}\right) \left(\sum_{n=0}^{\infty}c(n) q^{\alpha n+\beta}\right)\pmod{2}.\]
Then by Lemma \ref{lem2}, we have that
\[f_{\alpha r + \beta,\alpha t}(z) \coloneqq \sum_{\alpha n+\beta\equiv \alpha r+\beta \pmod*{\alpha t}}a_t(n)q^{\alpha n+\beta} \in S_{2^j\cdot 12+k}\left(\Gamma_1\left(\frac{N\alpha^2t^2}{d}\right)\right)\]
and, moreover,
\[f_{\alpha r + \beta,\alpha t}(z) = \sum_{n\equiv r\pmod*{t}}a_t(n)q^{\alpha n+\beta} \equiv \left(\sum_{n=0}^\infty q^{2^j\alpha t(2n+1)^2} \right)\left(\sum_{n\equiv r\pmod*{t}} c(n) q^{\alpha n + \beta}\right) \pmod{2}.\]
If $c(M)$ is odd for at least one $M\equiv r\pmod{t}$ but for only finitely many, then this factorization modulo 2 contradicts Lemma \ref{lem1}. This proves the first statement of the theorem.

To prove the second statement of the theorem, suppose that $c(M)$ is even for all $M\equiv r\pmod{t}$ where $0\leq M\leq C_{r,t}.$ It follows that $c(M)$ is even for all $M\equiv r\pmod{t}$ where
\[r \leq M \leq \frac{2^j\cdot 12+k}{12\alpha}\left[\frac{N\alpha^2t^2}{d}\right]^2\prod_{p\mid N\alpha t} \left(1 - \frac{1}{p^2}\right)-2^jt+r.\]
Thus the first odd term of the factor $\displaystyle{\sum_{n\equiv r\pmod*{t}} c(n) q^{\alpha n + \beta}}$ has exponent at least
\[\frac{2^j\cdot 12+k}{12}\left[\frac{N\alpha^2t^2}{d}\right]^2\prod_{p\mid N\alpha t} \left(1 - \frac{1}{p^2}\right)-2^j\alpha t+\alpha(t+r) +\beta.\]
Noting that the first odd term of the factor $\sum_{n=0}^\infty q^{2^j\alpha t(2n+1)^2}$ has exponent $2^j\alpha t,$ we find that 
\begin{align*}
\mathrm{Ord}_2(f_{\alpha r + \beta,\alpha t}(z))
&\geq \frac{2^j\cdot 12+k}{12}\left[\frac{N\alpha^2t^2}{d}\right]^2\prod_{p\mid N\alpha t} \left(1 - \frac{1}{p^2}\right)+\alpha(t+r) +\beta\\
&> \frac{2^j\cdot 12+k}{12}\left[\frac{N\alpha^2t^2}{d}\right]^2\prod_{p\mid N\alpha^2t^2/d} \left(1 - \frac{1}{p^2}\right)\end{align*}
By Lemma \ref{sturm}, this implies that $f_{\alpha r+\beta, \alpha t}(z)\equiv 0\pmod{2},$ and thus $c(M)$ is even for all $M\equiv r\pmod{t}.$
\end{proof}


\section{The parity of $\frac{\cphibar{2}(n)}{4}$} \label{cors}

Recall that we defined $\cphibar{2}(n)\coloneqq\cphi{2}(n)-p(n)$. In fact, this is a special case of a function $\overline{c\phi_k}(n),$ which was defined combinatorially by Kolitsch \cite{Kolitsch89, Kolitsch91}, who also found congruences for this function. For example, he proved the following generalization of equation \eqref{mod4cong}:
\[\overline{c\phi_k}(n) \equiv 0\pmod{k^2}.\]
This was strengthened by Sellers \cite{Sellers94}, who proved congruences modulo higher powers of 2 and 3. For instance, we have that
\[\overline{c\phi_2}(2n)\equiv 0\pmod{8}.\]
This fact as well as results of Cui et al.\ \cite{CuiGuHuang16} provide counterexamples to the analogue of Subbarao's conjecture for $\cphibar{2}(n)/4.$ However, in this section, we use Theorems \ref{mainthm1gen} and \ref{mainthm2gen} to prove Theorems \ref{mainthm1} and \ref{mainthm2}, which are analogous to theorems of Ono that give strong results on the parity of $p(n).$

\begin{proof}[Proof of Theorems \ref{mainthm1} and \ref{mainthm2}]
In \cite{Sellers94}, Sellers proved that the generating function for $\cphibar{2}(n)$ is \[\sum_{n=0}^\infty \cphibar{2}(n)q^{n}=4q\prod_{n=1}^\infty\frac{(1-q^{16n})^2}{(1-q^n)^2(1-q^{8n})}.\]
Using the fact that $\frac{\eta(z)^2}{\eta(2z)}\equiv 1\pmod*{2},$ we then have
\[\sum_{n=0}^\infty\frac{\cphibar{2}(n)}{4}q^{12n-1}=\frac{\eta(192z)^2}{\eta(12z)^2\eta(96z)}\equiv\frac{\eta(192z)^2\eta(z)^2}{\eta(12z)^2\eta(96z)\eta(2z)}\pmod{2}.\]
Using standard results regarding the modularity properties of eta-quotients \cite{GordonHughes93, Newman57, Newman58}, one can check that $\frac{\eta(192z)^2\eta(z)^2}{\eta(12z)^2\eta(96z)\eta(2z)} \in M_{0}^!(\Gamma_0(576),\leg{12}{\bul}).$

Then by Proposition \ref{prop1} (together with standard results on the order of vanishing of eta-quotients at cusps \cite{Ligozat75}) it follows that for any integer $j$ with $2^j>\frac{t}{12}$ we have that
\[f_t(z)\coloneqq \frac{\eta(192z)^2\eta(z)^2}{\eta(12z)^2\eta(96z)\eta(2z)} \Delta^{2^j}(12 tz)\in S_{12\cdot2^j}\left(\Gamma_0(576t),\leg{12t^2}{\bul}\right).\]
Finally, our desired results follow immediately by applying Theorems \ref{mainthm1gen} and \ref{mainthm2gen} to $\sum \frac{\cphibar{2}(n)}{4}q^{12n-1}.$
\end{proof}


\section{Acknowledgements}
The authors would like to acknowledge and thank Matthew Boylan, Jeremy Rouse, and James Sellers for their insightful discussions that helped to shape this project.

\bibliography{refs.bib}{}
\bibliographystyle{plain}
\end{document}